\documentclass[12pt,a4paper,oneside]{amsart}
\usepackage{amsmath,amstext,amsthm,amssymb,amsxtra,asymptote,txfonts}
\usepackage[left=2.5cm,right=2.5cm,bottom=2cm,top=2cm]{geometry} 

\usepackage{txfonts,pxfonts,tikz} 
\usepackage[T1]{fontenc}

\newtheorem{theorem}{Theorem}

\newtheorem{lemma}{Lemma}

\newtheorem{open}{Open Problem}

\newtheorem{definition}{Definition}[section]
\theoremstyle{definition}

\newcommand{\beql}[1]{\begin{equation}\label{#1}}
\newcommand{\eeq}{\end{equation}}
\newcommand{\comment}[1]{}

\newcommand{\Abs}[1]{{\left|{#1}\right|}}

\newcommand{\Norm}[1]{{\left\|{#1}\right\|}}

\newcommand{\Set}[1]{{\left\{{#1}\right\}}}

\newcommand{\RR}{{\mathbb R}}

\newcommand{\CC}{{\mathbb C}}
\newcommand{\ZZ}{{\mathbb Z}}
\newcommand{\NN}{{\mathbb N}}

\newcommand{\inner}[2]{{\langle #1, #2 \rangle}}

\newcommand{\dens}{{\rm dens\,}}

\newcommand{\sgn}{{\rm sgn\,}}

\newcommand{\e}[1]{{e\left({#1}\right)}}

\newcounter{rem}
\newcounter{step}
\newcounter{mysec}
\newcounter{mysubsec}[mysec]
\newcounter{othm}
\def\theothm{\Alph{othm}}

\newcounter{fcap}
\newcommand{\mycaption}[2]{
 \refstepcounter{fcap}\label{#2}
 Figure \ref{#2}: {#1}
}

\begin{document}

\title{Multiple lattice tiles and Riesz bases of exponentials}

\author[M. Kolountzakis]{{Mihail N. Kolountzakis}}
\address{M.K.: Department of Mathematics, University of Crete, GR-700 13, Iraklio, Greece}
\email{kolount@math.uoc.gr}

\date{\today}

\begin{abstract}
Suppose $\Omega\subseteq\RR^d$ is a bounded and measurable set and $\Lambda \subseteq \RR^d$ is
a lattice. Suppose also that $\Omega$ tiles multiply, at level $k$, when translated at the locations
$\Lambda$. This means that the $\Lambda$-translates of $\Omega$ cover almost every point of $\RR^d$
exactly $k$ times. We show here that there is a set of exponentials $\exp(2\pi i t\cdot x)$, $t\in T$,
where $T$ is some countable subset of $\RR^d$, which forms a Riesz basis of $L^2(\Omega)$.
This result was recently proved by Grepstad and Lev under the extra assumption that $\Omega$ has
boundary of measure $0$, using methods from the theory of quasicrystals. Our approach is rather more
elementary and is based almost entirely on linear algebra.
The set of frequencies $T$ turns out to be a finite union of shifted copies of the dual
lattice $\Lambda^*$. It can be chosen knowing only $\Lambda$ and $k$ and is the same for all $\Omega$
that tile multiply with $\Lambda$.
\end{abstract}

\maketitle


\ 


\tableofcontents 

\noindent{\bf Notation: } We write $\e{x} = e^{2\pi i x}$. If $A$ is a set then $\chi_A$ is its
indicator function. If $\Lambda=A\ZZ^d$ is a lattice in $\RR^d$ then
$\Lambda^*=A^{-\top}\ZZ^d$ denotes the dual lattice.

\section{Introduction}\label{sec:intro}

\subsection{Riesz bases}\label{sec;riesz-bases}
In this paper we deal with the question of existence of a Riesz (unconditional) basis of exponentials
$$
e_t(x) := e(t \cdot x)= e^{2\pi i t \cdot x},\ \  t \in L,
$$
for the space $L^2(\Omega)$, where $\Omega\subseteq\RR^d$ is a domain of finite Lebesgue measure and $L \subseteq \RR^d$ is a countable set.
By Riesz basis we mean that every $f\in L^2(\Omega)$ can be written uniquely in the form
\beql{riesz-expansion}
f(x) = \sum_{t \in L} a_t \cdot e(t \cdot x)
\eeq
with the coefficients $a_t$ satisfying
\beql{riesz-bounds}
C_1 \Norm{f}_2^2 \le \sum_{t \in L} \Abs{a_t}^2 \le C_2 \Norm{f}_2^2,
\eeq
for some positive and finite constants $C_1, C_2$.

\subsection{Orthogonal bases}\label{sec:orthogonal-bases}
One very special example of a Riesz basis occurs
when the exponentials $e(t\cdot x), t\in L$,
can be chosen to be orthogonal and complete for $L^2(\Omega)$. One can then choose
$a_t = \Abs{\Omega}^{-1/2}\inner{f}{e_t}$ and $C_1=C_2=\Abs{\Omega}$ for
\eqref{riesz-bounds} to hold as an equality.
For instance, if $\Omega=(0,1)^d$ is the unit cube in $\RR^d$ then one can take $L=\ZZ^d$ and obtain such an orthogonal basis of exponentials.
This case, where an orthogonal basis of exponentials exists, is a very rigid
situation though and many ``reasonable'' domains do not have
such a basis (a ball is one example, or any other smooth convex body or any non-symmetric convex body).

The problem of which domains admit an orthogonal basis of exponentials
has been studied intensively.
The so called Fuglede or Spectral Set Conjecture \cite{fuglede1974operators} (claiming that for $\Omega$ to have such a basis it is necessary and sufficient that it can tile space by translations)
was eventually proved to be false in dimension at least 3
\cite{tao2004fuglede,kolountzakis2006tiles,farkas2006onfuglede,farkas2006tiles}, in both
directions. Yet the conjecture may still be true in several important special cases such as
convex bodies \cite{iosevich2003fuglede}, and it generated many interesting results even after
the disproof of its general validity (a rather dated account may be found in \cite{kolountzakis2004milano}).

It is expected that the existence of a Riesz basis for a domain $\Omega$ is a much
more general, and perhaps even generic, phenomenon, although proofs of existence of a Riesz basis for specific domains are still rather rare, especially in higher dimension
\cite{kozma2012combining,lyubarskii2000complete,marzo2006riesz}.

\subsection{Lattice tiles}\label{sec:lattice-tiles}
One general class of domains for which an orthogonal basis of
exponentials is known to exist is the class of {\em lattice tiles}.
A domain $\Omega\in\RR^d$ is said to {\em tile} space when translated at the locations of the lattice $L$ (a discrete additive subgroup of $\RR^d$ containing $d$ linearly independent vectors) if
\beql{lattice-tiling}
\sum_{t\in L} \chi_\Omega(x-t) = 1,\ \ \mbox{for almost all $x\in\RR^d$}.
\eeq
Intuitively this condition means that one can cover $\RR^d$ with the $L$-translates
of $\Omega$, with no overlaps, except for a set of measure zero (usually the translates
of $\partial\Omega$, for ``nice'' domains $\Omega$).

It is not hard to see that when $\Omega$ has finite and non-zero
measure then the set $L$ has density equal to $1/\Abs{\Omega}$.
If $L$ is a lattice then we call $\Omega$ an {\em almost fundamental domain} of $L$ and $\Abs{\Omega}=(\dens{L})^{-1}$.
A {\em fundamental domain} of $L$ is any set which contains
exactly one element of each coset mod $L$,
for instance a fundamental parallelepiped.
There are of course many others, as indicated in Figure \ref{fig:tile}.

\begin{figure}[h]
\begin{center}
\begin{asy}
size(8cm);

pair a=(0,0), b=(1,0), c=(1,0.7), d=(0.3,1), e=(0,1), c1=(2,0.7), d1=(1.3,1), f=(2,1);

draw((-0.5,0) -- (2.5,0), dashed);
draw((-0.5,1) -- (2.5,1), dashed);
draw((0,-0.5)-- (0,1.2), dashed);
draw((1,-0.5)-- (1,1.2), dashed);
draw((2,-0.5)-- (2,1.2), dashed);

dot(a); dot(b); dot(c); dot(d); dot(e); dot(c1); dot(d1); dot(f);

fill(a -- b -- c -- d -- e -- cycle, mediumgray);
fill(c1 -- f -- d1 -- cycle, mediumgray);
label("0", (0,0), SW); label("1", (1,0), SE); label("1",(0,1), SW);
label("$\Omega$", (0.5,0.5), SW);
label("$\Omega$", (1./3.)*(c1+f+d1), E);

\end{asy}

\mycaption{Shaded $\Omega$ is a fundamental domain of $\RR^2/\ZZ^2$}{fig:tile}
\end{center}
\end{figure}

Every lattice tile by the lattice $L$ has an orthogonal basis of exponentials, namely
those with frequencies $t \in L^*$, where $L^*$ is the dual lattice \cite{fuglede1974operators}.

\subsection{Multiple tiling by a lattice}\label{sec:multiple-tiling}
We say that a domain tiles multiply when its translates cover space the same number
of times, almost everyhwere.
\begin{definition}
Let $\Omega \subseteq \RR^d$ be measurable and $L \subseteq \RR^d$ be a countable set.
We say that $\Omega$ tiles $\RR^d$ when translated by $L$ at level $k \in \NN$ if
\beql{tiling}
\sum_{t\in L} \chi_\Omega(x-t) = k,
\eeq
for almost every $x\in\RR^d$.
If we do not specify $k$ then we mean $k=1$. 
\end{definition}
Multiple tiles are a much wider class of domains that level-one tiles.
For instance, any centrally symmetric convex polygon in the plane
whose vertices have integer coordinates
tiles multiply by the lattice $\ZZ^2$ at some level $k \in \NN$.
In contrast, only parallelograms or symmetric hexagons can tile at level one.

Another difference is the fact that if two {\em disjoint} domains $\Omega_1$ and $\Omega_2$
both tile multiply when translated at the locations $L$ then so does their union.
In the case of multiple lattice tiling this operation gives essentially the totality of
multiple tiles starting from level-one tiles, according to the following easy Lemma.

\begin{lemma}\label{lm:split}
Suppose $\Omega\subseteq\RR^d$ is a measurable set which
tiles $\RR^d$ at level $k$ when
translated by the lattice $\Lambda\subseteq\RR^d$. Then we can write
\beql{split}
\Omega = \Omega_1 \cup \cdots \cup \Omega_k \cup E,
\eeq
where $E$ has measure $0$ and the $\Omega_j$ are measurable, mutually disjoint
and each $\Omega_j$ is an almost fundamental domain of the lattice $\Lambda$.
\end{lemma}

\begin{proof}
Let $D\subseteq\RR^d$ be a measurable fundamental domain of $\Lambda$,
for instance one of its fundamental parallelepipeds.
For almost every $x \in D$ (call the exceptional set $E\subseteq D$)
it follows from our tiling assumption
that $\Omega \cap (x+\Lambda)$ contains exactly $k$ points,
which we denote by
$$
p_1(x) < p_2(x) < \cdots < p_k(x),
$$
ordered according to the lexicographical ordering in $\RR^d$.
We also have that almost every point of $\Omega$ belongs to exactly one such list.

Let then $\Omega_j = \bigcup_{x\in D\setminus E} p_j(x)$, for $j=1,2,\ldots,k$.
In other words, for (almost) each one of the classes
mod $\Lambda$ we distribute its $k$ occurences
in $\Omega$ into the sets $\Omega_j$.
It is easy to see that the $\Omega_j$ are disjoint and measurable and
that they are almost fundamental domains of $\Lambda$.
\end{proof}

\subsection{Multiple lattice tiles have Riesz bases of exponentials}
\label{sec:main-result-description}
It is not true that domains that tile multiply by a lattice have an orthogonal basis
of exponentials. For instance, it is known \cite{iosevich2003fuglede}
that the only convex polygons that have
such a basis are parallelograms and symmetric hexagons, yet every symmetric convex polygon
with integer vertices is a multiple tile, a much wider class.

It is however true that multiple tiles have a Riesz basis of exponentials.
The main result of this paper is the following theorem.
\begin{theorem}\label{th:main}
Suppose $\Omega\subseteq\RR^d$ is bounded, measurable and
tiles $\RR^d$ multiply at level $k$ with the lattice $\Lambda$.
Then there are vectors $a_1,\ldots,a_k \in \RR^d$ such that the exponentials
\beql{the-basis}
\e{(a_j+\lambda^*)\cdot x},\ j=1,2,\ldots,k,\, \lambda^*\in\Lambda^*
\eeq
form a Riesz basis for $L^2(\Omega)$.

The vectors $a_1,\ldots,a_k$ depend on $\Lambda$ and $k$ only, not on $\Omega$.
\end{theorem}
Theorem \ref{th:main} was proved by Grepstad and Lev \cite{grepstad2012multi}
with the additional topological assumption that the boundary $\partial\Omega$ has Lebesgue
measure 0.

In \cite{grepstad2012multi} the result is proved following the method of \cite{matei2010simple,matei2008quasicrystals} on quasicrystals.
Our approach is more elementary and almost entirely based on linear algebra.
The authors of \cite{grepstad2012multi} have pointed out to me that there are
similarities of the method in this paper and the methods in
\cite{lyubarskii2000complete,lyubarskii1997sampling,marzo2006riesz}.

As an interesting corollary of Theorem \ref{th:main} let us mention, as is done in
\cite{grepstad2012multi}, that, according to the recent result of \cite{gravin2011translational}, if $\Omega$ is a centrally symmetric polytope in $\RR^d$,
whose codimension 1 faces are also centrally symmetric and whose vertices all have rational
coordinates, then $L^2(\Omega)$ has a Riesz basis of exponentials.

\begin{open}
Is Theorem \ref{th:main} still true if $\Omega$ is of finite measure but unbounded?
\end{open}


\section{Proof of the main result}\label{sec:proof}
The essence of the proof is contained in the following lemma.
\begin{lemma}\label{lm:sum}
Suppose $\Omega\subseteq\RR^d$ is bounded, measurable and
tiles $\RR^d$ multiply at level $k$ with the lattice $\Lambda$.
Then there exist vectors $a_1, a_2, \ldots, a_k \in \RR^d$
such that the following is true.

For any $f \in L^2(\Omega)$ there are
unique measurable functions $f_j:\RR^d\to\CC$ such that
\begin{enumerate}
\item The $f_j$ are $\Lambda$-periodic,
\item The $f_j$ are in $L^2$ of any almost fundamental domain of $\Lambda$, and
\item We have the decomposition
\beql{sum}
f(x) = \sum_{j=1}^k \e{a_j\cdot x} f_j(x),\ \ \mbox{for a.e.\ $x\in\Omega$}.
\eeq
\end{enumerate}
Finally we have
\beql{norm-control}
C_1 \Norm{f}_{L^2(\Omega)}^2 \le \sum_{j=1}^k \Norm{f_j}_{L^2(\Omega)}^2
  \le C_2 \Norm{f}_{L^2(\Omega)}^2\ ,
\eeq
where $0<C_1, C_2<\infty$ depend only on $\Omega$ and not on $f$.
\end{lemma}

\begin{proof}
Using Lemma \ref{lm:split} we can write $\Omega$ as the disjoint union
$$
\Omega = \Omega_1 \cup \cdots \cup \Omega_k,
$$
where each $\Omega_k$ is a measurable almost fundamental domain of $\Lambda$.
We can now define for $j=1,2,\ldots,k$ and for almost every $x \in \RR^d$
\beql{omega-x}
\omega_j(x)\ \mbox{as the unique point in $\Omega_j$ s.t.\ $x-\omega_j(x) \in \Lambda$, and}
\eeq
\beql{lambda-x}
\lambda_j(x) = x - \omega_j(x).
\eeq
(The maps $\omega_j$ are clearly measurable and measure-preserving when restricted to a fundamental
domain of $\Lambda$.)
Since the sought-after $f_j$ are to be $\Lambda$-periodic it is enough to define them
on $\Omega_1$ and extend them to $\RR^d$ by their $\Lambda$-periodicity.
We may therefore rewrite our target decomposition \eqref{sum} equivalently as follows.
\beql{sum1}
\mbox{For each $x\in \Omega_1$ and $r=1,2,\ldots,k$:
	$f(\omega_r(x)) = \sum_{j=1}^k \e{a_j\cdot (x-\lambda_r(x))} f_j(x)$.}
\eeq
We view \eqref{sum1} as a $k\times k$ linear system
\beql{lin-sys}
M \widetilde{F} = F
\eeq
whose right-hand side is the column vector
$$
F = (f(\omega_1(x)), f(\omega_2(x)), \ldots, f(\omega_k(x)))^\top
$$
and the unknowns form the column vector
$$
\widetilde{F} = (f_1(x), f_2(x), \ldots, f_k(x))^\top.
$$
We have a different linear system for each $x \in \Omega_1$ and its matrix is
$M = M(x) \in \CC^{k\times k}$ with
\beql{matrix-m}
M_{r,j} = M_{r,j}(x) = \e{a_j\cdot(x-\lambda_r(x))},\ \ r,j=1,2,\ldots,k.
\eeq
Factoring we can write this matrix as
\beql{matrix-n}
M(x) = N(x) \, {\rm diag\,}(\e{a_1\cdot x}, \e{a_2\cdot x}, \ldots, \e{a_k\cdot x}),
\eeq
with the matrix $N = N(x)$ given by
$$
N_{r,j} = N_{r,j}(x) = \e{-a_j\cdot\lambda_r(x)},\ \ r,j=1,2,\ldots,k.
$$
The key observation here is that when varying $x \in \Omega_1$
the number of different $N(x)$ matrices that arise
(the $a_j$ are fixed) is finite and depends on $\Omega$ only.
The reason for this is that the vectors $\lambda_r(x)$
are among the $\Lambda$ vectors in the bounded set $\Omega-\Omega$, hence they take values in  a finite set. (This is the only place where the boundedness of $\Omega$ is used.)

Let us now see that the vectors $a_1,\ldots,a_k$ can be chosen so that all the (finitely many) possible matrices $N$ are invertible. We have
\beql{det-n}
\det{N(x)} = \sum_{\pi \in S_k} \sgn(\pi)\,\e{-\sum_{j=1}^k a_j\cdot\lambda_{\pi_j}(x)},
\eeq
where $S_k$ denotes the permutation group on $\Set{1,2,\ldots,k}$.
By the definition of the vectors $\lambda_r(x)$
and the disjointness of the sets $\Omega_r$ it follows that for each $x$
no two $\lambda_r(x)$ can be the same.
View now the expression \eqref{det-n} as a function
of the vector $a=(a_1,\ldots,a_k) \in \RR^{dk}$.
Clearly it is a trigonometric polynomial and
it is not identically zero as all the frequencies
(for $\pi$ in the symmetric group $S_k$)
\beql{freqs}
\lambda_\pi(x) = (\lambda_{\pi_1}(x), \ldots, \lambda_{\pi_k}(x)) \in \RR^{dk},
\eeq
are distinct precisely because all the $\lambda_r(x)$ are distinct.
Since the zero-set of any trigonometric polynomial (that is not identically zero)
is a set of codimension at least 1 it follows that the vectors
$a_1,\ldots,a_k$ can be chosen so that all the $N(x)$ matrices that arise are invertible.

Let now $x\in\Omega_1$ and consider the solution of the linear system
\eqref{lin-sys} at $x$
that now takes the form
\beql{lin-sys-n}
\widetilde{F}(x) =
 {\rm diag\,}(\e{-a_1\cdot x}, \e{-a_2\cdot x}, \ldots, \e{-a_k\cdot x}) \, N(x)^{-1} F(x).
\eeq
Since $N(x)$ runs through a finite number of invertible matrices it follows that there
are finite constants $A_1, A_2>0$, independent of $f$,
such that for any $x\in\Omega_1$ we have
\beql{norms}
A_1 \Norm{F(x)}_{\ell^2}^2 \le \Norm{\widetilde{F}(x)}_{\ell^2}^2
   \le  A_2 \Norm{F(x)}_{\ell^2}^2.
\eeq
Integrating \eqref{norms} over $\Omega_1$ we obtain
\beql{norm-bounds}
A_1 \Norm{f}_{L^2(\Omega)}^2 \le \sum_{j=1}^k \Norm{f_j}_{L^2(\Omega_1)}^2
 \le A_2 \Norm{f}_{L^2(\Omega)}^2.
\eeq
This implies \eqref{norm-control} with $C_j=k\cdot A_j$, $j=1,2$.
To show the uniqueness of the decomposition \eqref{sum} observe that any such decomposition
must satisfy the linear system \eqref{lin-sys-n}, whose non-singularity has been
ensured by our choice of the $a_j$.
\end{proof}

We can now complete the proof of our main result.

\begin{proof}[Proof of Theorem \ref{th:main}]
Let $f\in L^2(\Omega)$. By Lemma \ref{lm:sum} we can write $f$ as in \eqref{sum}.
Since the $f_j$ are $\Lambda$-periodic and are in $L^2$ of any almost fundamental
domain $D$ of $\Lambda$ it follows that we can expand each $f_j$ in the frequencies
of $\Lambda^*$ (the dual lattice of $\Lambda$)
\beql{fj-expand}
f_j(x) = \sum_{\lambda^*\in\Lambda^*} f_{j,\lambda^*} \e{\lambda^*\cdot x},
\ j=1,2,\ldots,k,
\eeq
with
\beql{fj-parseval}
\Norm{f_j}_{L^2(D)}^2 = \sum_{\lambda^*\in\Lambda^*} \Abs{f_{j,\lambda^*}}^2,
\eeq
since the exponentials $e(\lambda^*\cdot x)$, $\lambda^*\in\Lambda^*$, form
an orthogonal basis of $L^2(D)$.

The completeness of \eqref{the-basis} follows from \eqref{sum}:
\beql{f-decomp}
f(x) = \sum_{j=1}^k \sum_{\lambda^*\in\Lambda^*} f_{j,\lambda^*} \e{(a_j+\lambda^*)\cdot x}.
\eeq
The fact that \eqref{the-basis} is a Riesz sequence follows from \eqref{norm-control}:
$$
\frac{1}{C_2 k} \sum_{j,\lambda^*} \Abs{f_{j,\lambda^*}}^2 \le 
 \Norm{ \sum_{j,\lambda^*} f_{j,\lambda^*} \e{(a_j+\lambda^*)\cdot x} }_{L^2(\Omega)}^2 \le 
 \frac{1}{C_1 k} \sum_{j,\lambda^*} \Abs{f_{j,\lambda^*}}^2.
$$

As is clear from the proof above, the $k$-tuples of vectors $a_1,\ldots,a_k$ that appear in Theorem \ref{th:main} are a generic choice: almost all $k$-tuples will do. The exceptional
set in $\RR^{dk}$ is a set of lower dimension.

With a little more care one can see that one can choose the vectors $a_1,\ldots,a_k$ to depend on $\Lambda$ and $k$ only and not on $\Omega$. In the proof of Lemma \ref{lm:sum} the $a_j$ were chosen
to ensure that the trigonometric polynomials \eqref{det-n} are all non-zero. Fix $\Lambda$ and $k$
and form the set of all polynomials of the form \eqref{det-n} which are not identically zero.
This set of polynomials is countable and each such polynomial vanishes on a set of
codimension at least 1 in $\RR^{dk}$.
It follows that the union of their zero sets cannot possibly exhaust $\RR^{dk}$
and we only have to choose the $a_j$ to avoid that union.

Thus there is a choice of $a_j$ that works for all $\Omega$ of the same lattice.
This proof does not give uniform values for the constants $C_1$ and $C_2$ in \eqref{norm-control} though.
\end{proof}

\bibliographystyle{abbrv}
\bibliography{spectral-sets}

\begin{thebibliography}{10}

\bibitem{farkas2006onfuglede}
B.~Farkas, M.~Matolcsi, and P.~M\'ora.
\newblock On {Fuglede}'s conjecture and the existence of universal spectra.
\newblock {\em J. Fourier Anal. Appl.}, 12(5):483--494, 2006.

\bibitem{farkas2006tiles}
B.~Farkas and S.~R{\'e}v{\'e}sz.
\newblock Tiles with no spectra in dimension 4.
\newblock {\em Math. Scand.}, 98(1):44--52, 2006.

\bibitem{fuglede1974operators}
B.~Fuglede.
\newblock Commuting self-adjoint partial differential operators and a group
  theoretic problem.
\newblock {\em J. Funct.\ Anal.}, 16:101--121, 1974.

\bibitem{gravin2011translational}
N.~Gravin, S.~Robins, and D.~Shiryaev.
\newblock Translational tilings by a polytope, with multiplicity.
\newblock {\em Combinatorica}, pages 1--21, 2011.

\bibitem{grepstad2012multi}
S.~Grepstad and N.~Lev.
\newblock {Multi-tiling and Riesz bases}.
\newblock {\em arXiv preprint arXiv:1212.4679}, 2012.

\bibitem{iosevich2003fuglede}
A.~Iosevich, K.~N., and T.~T.
\newblock {The Fuglede spectral conjecture holds for convex bodies in the
  plane}.
\newblock {\em Math.\ Res.\ Letters}, 10:559--570, 2003.

\bibitem{kolountzakis2004milano}
M.~Kolountzakis.
\newblock {The study of translational tiling with Fourier Analysis}.
\newblock In L.~Brandolini, editor, {\em Fourier Analysis and Convexity}, pages
  131--187. Birkh\"auser, 2004.

\bibitem{kolountzakis2006tiles}
M.~Kolountzakis and M.~Matolcsi.
\newblock {Tiles with no spectra}.
\newblock {\em Forum Math.}, 18:519--528, 2006.

\bibitem{kozma2012combining}
G.~Kozma and S.~Nitzan.
\newblock {Combining Riesz bases}.
\newblock {\em arXiv preprint arXiv:1210.6383}, 2012.

\bibitem{lyubarskii2000complete}
Y.~I. Lyubarskii and A.~Rashkovskii.
\newblock Complete interpolating sequences for fourier transforms supported by
  convex symmetric polygons.
\newblock {\em Arkiv f{\"o}r matematik}, 38(1):139--170, 2000.

\bibitem{lyubarskii1997sampling}
Y.~I. Lyubarskii and K.~Seip.
\newblock Sampling and interpolating sequences for multiband-limited functions
  and exponential bases on disconnected sets.
\newblock {\em Journal of Fourier Analysis and Applications}, 3(5):597--615,
  1997.

\bibitem{marzo2006riesz}
J.~Marzo.
\newblock {Riesz basis of exponentials for a union of cubes in ${\mathbb
  R}^d$}.
\newblock {\em arXiv preprint math/0601288}, 2006.

\bibitem{matei2008quasicrystals}
B.~Matei and Y.~Meyer.
\newblock Quasicrystals are sets of stable sampling.
\newblock {\em Comptes Rendus Mathematique}, 346(23):1235--1238, 2008.

\bibitem{matei2010simple}
B.~Matei and Y.~Meyer.
\newblock Simple quasicrystals are sets of stable sampling.
\newblock {\em Complex Variables and Elliptic Equations}, 55(8-10):947--964,
  2010.

\bibitem{tao2004fuglede}
T.~Tao.
\newblock {Fuglede's conjecture is false in 5 and higher dimensions}.
\newblock {\em Math. Res. Lett.}, 11(2-3):251--258, 2004.

\end{thebibliography}

\end{document}